\documentclass{SCAE}
\numberwithin{equation}{section}
\begin{document}

\Year{2016} %
\Month{January}
\Vol{59} %
\No{1} %
\BeginPage{1} %
\EndPage{XX} %
\AuthorMark{Last1 F N {\it et al.}} \ReceivedDay{November 17, 2014}
\AcceptedDay{January 22, 2015} \PublishedOnlineDay{; published
online January 22, 2016}
\DOI{ 10.1007/s11425-016-0171-4} 

\title[Stochastic Calculus with respect to $G$-Brownian Motion Viewed through Rough Paths]{Stochastic Calculus with respect to $G$-Brownian Motion Viewed through Rough Paths}{}


\author[1,2]{Peng, ShiGe}{Corresponding author}
\author[1,2]{Zhang, HuiLin}{}

\address[{\rm1}]{School of Mathematics, Shandong University, Jinan {\rm 250100}, China;}
\address[{\rm2}]{Institute for Advanced Research, Shandong University, Jinan {\rm 250100}, China;}
 \Emails{ peng@sdu.edu.cn,
huilinzhang2014@gmail.com}\maketitle


 {\begin{center}
\parbox{14.5cm}{\begin{abstract}
 In this paper, we study rough path
properties of stochastic integrals of It\^{o}'s type and
Stratonovich's type with respect to $G$-Brownian motion. The
roughness of $G$-Brownian Motion is estimated and then the pathwise
Norris lemma in $G$-framework is obtained.\vspace{-3mm}
\end{abstract}}\end{center}}



\renewcommand{\baselinestretch}{1.2}
\begin{center} \renewcommand{\arraystretch}{1.5}
{\begin{tabular}{lp{0.8\textwidth}} \hline \scriptsize {\bf
Citation:}\!\!\!\!&\scriptsize Peng, S G, Zhang, H L.
\makeatletter\@titlehead. Sci China Math, 2016, 59,
 doi:~\@DOI\makeatother\vspace{1mm}
\\
\hline
\end{tabular}}\end{center}

\baselineskip 11pt\parindent=10.8pt  \wuhao

\newtheorem{thm}{Theorem}[section]
\newtheorem{assumption}[thm]{Assumption}
\newtheorem{lem}[thm]{Lemma}
\newtheorem{coro}[thm]{Corollary}
\newtheorem{rem}[thm]{Remark}
\newtheorem{prop}[thm]{Proposition}

\newtheorem{defi}[thm]{Definition}
\newcommand{\EC}{\hat{\mathbb{E}}_\mathbb{C}}
\newcommand{\HC}{\mathcal {H}_\mathbb{C}}
\newcommand{\OHEC}{(\Omega,\,\HC,\,\EC)}
\newcommand{\E}{\hat{\mathbb{E}}}
\newcommand{\R}{\mathbb{R}}
\newcommand{\OHE}{(\Omega,\,\mathcal{H},\,\E)}
\newcommand{\X}{\mathbb{X}}
\newcommand{\BX}{\mathbf{X}}
\newcommand{\K}{\mathbb{K}}
\newcommand{\oc}{\mathcal {C}}
\newcommand{\FC}{\mathscr{C}}
\newcommand{\B}{\mathbb{B}}
\newcommand{\BB}{\mathbf{B}}
\newcommand{\op}{\mathcal{P}}
\newcommand{\FD}{\mathscr{D}}
\newcommand{\oq}{\mathcal{Q}}
\newcommand{\oor}{\mathcal {R}}
\newcommand{\hc}{\hat{c}}
\newcommand{\vep}{\varepsilon}
\newcommand{\tX}{\tilde{X}}
\newcommand{\te}{\tilde{\vep}}
\newcommand{\tb}{\tilde{b}}
\newcommand{\tA}{\tilde{A}}

\section{Introduction}

Since Pardoux and Peng \cite{PP90}, backward stochastic differential
equations(BSDEs) receive much attention and are widely applied in
many areas such as stochastic control, financial mathematics, PDEs
(see \cite{KPQ95}, \cite{MY97}, \cite{PP94}, \cite{P90}, \cite{P97}
for example). However, BSDEs fail to give a probabilistic
explanation to fully nonlinear PDEs. Motivated by such disadvantages
of BSDEs and applications in financial mathematics, $G$-expectation
theory was introduced by Peng in \cite{P07a}, \cite{P08a}.
$G$-expectation is a time consistent sublinear expectation, which is
obtained from a fully nonlinear parabolic PDE, called $G$-heat
equation, with the canonical process $B_t$ as $G$-Brownian motion.
Stochastic analysis and the corresponding BSDEs in $G$-framework are
established in \cite{P07a}, \cite{P08a},
\cite{P10}, \cite{HJPS14a}, \cite{HJPS14b}.\\

Rough path theory was introduced by Lyons in his pioneer work
\cite{L98}, to give a well defined integration when the driving path
is not smooth (with $p$-variation for $p \geq 2$). The universal
limit theorem for differential equations driven by rough paths was
obtained and the continuity of It\^{o}-Lyons map for the
corresponding rough differential equations (RDEs for short) was
firstly established by Lyons. Later, Gubinelli expanded integrands
of rough integral from one-forms to controlled paths(see \cite{G04},
\cite{G10}). Geng et al first investigated rough path properties of
$G$-Brownian motion in \cite{GQY14}. Firstly, $G$-Brownian motion is
lifted as geometric rough paths. Then, some basic relations between
SDEs and RDEs driven by $G$-Brownian motion were established. These
results allowed them to prove the existence and uniqueness theorem
of SDEs driven by $G$-Brownian motion on differentiable manifolds.\\

A natural question is what is the relation between rough integrals
and stochastic integrals with respect to $G$-Brownian motion.
Furthermore, does the $G$-Brownian motion possesses the roughness
pathwisely? In this paper we study the rough path properties based
on the $\alpha$-H\"{o}lder continuity of $G$-Brownian motion, of
which the enhancement could be completed by a generalized
Kolmogorov's criterion for rough paths under $G$-expectation
framework, which is more direct and probabilistic compared with
\cite{GQY14}. Moreover, the cross variation of It\^{o}'s process
under $G$-Brownian motion framework is studied, through which the
Stratonovich integral is defined. Then, the relation among rough
integral, It\^{o} integral and Stratonovich integral with respect to
$G$-Brownian motion is established. At last, the roughness of
$G$-Brownian motion is calculated and then the Norris lemma for
stochastic integral with respect to $G$-Brownian motion is obtained.
Further work about applications in finance such as no
arbitrage hedging and superhedging could be possibly available in later papers by authors. \\

The paper is organized as follows. In Section 2, we recall some
basic definitions and results in $G$-expectation theory and rough
path theory. Then in Section 3, $G$-Brownian motion is lifted as
rough paths, and It\^{o} integral with respect to $G$-Brownian
motion is proved to be equivalent to the corresponding rough
integral. Then, the quadratic variation of $G$-It\^{o} process is
introduced and the Stratonovich integral with respect to
$G$-Brownian motion is defined. Similarly, the equivalence between
$G$-Stratonovich integral and the corresponding rough integral is
established. In Section 4, the $\theta$-H\"{o}lder roughness of
$G$-Brownian motion is studied, and then the pathwise Norris lemma
in $G$-framework is obtained.

\section{Preliminaries about the $G$-expectation and Rough Paths}
In this part, we give some definitions and results of
$G$-expectation and rough path theories. The proofs can be found in
\cite{FH14}, \cite{LQ02}, \cite{P07a}, \cite{P10}.


\subsection{The rough path theory}

For rough path theory presented in this paper, we adopt the
framework of Friz and Hairer \cite{FH14}, see also Gubinelli
\cite{G04}. \\

Denote by $\R^m \otimes \R^n$ the algebraic tensor of two Euclidean
spaces. For any path on some interval $[0,T]$ with values in a
$\R^d,$ its $\alpha$-H\"{o}lder norm(semi-norm) is defined by
$$
\|X\|_\alpha =\sup_{0\leq s <t \leq T}
\frac{|X_{s,t}|}{|t-s|^\alpha},
$$
where $X_{s,t}=X_t - X_s,$ for any path $X.$\\

 Denote $\oc^\alpha([0,T],\R^d)$ as the space of paths with finite
$\alpha$-H\"{o}lder norm and values in $\R^d.$ Similarly, a mapping
$\mathbb{X}$ from $[0,T]^2$ to $\R^d\otimes \R^d$ is attached with
norm
$$
\|\X\|_{2\alpha}= \sup_{0\leq s \neq t \leq T}
\frac{|\X_{s,t}|}{|t-s|^{2\alpha}},
$$
whenever it's finite.\\

A rough path on some interval $[0,T]$ with values in $\R^d$ includes
a ``rough'' continuous path $X:[0,T] \rightarrow \R^d,$ along with
its ``iterated integration'' part $\mathbb{X}:[0,T]^2 \rightarrow
\R^d\otimes \R^d,$ which satisfies ``Chen's identity'',
\begin{equation}\label{chen}
\X_{s,t}-\X_{s,u}-\X_{u,t}= X_{s,u} \otimes X_{u,t},
\end{equation}
and H\"{o}lder continuity.

In the sequel, suppose $\alpha \in (\frac13,\frac12)$ for the need of rough integral with respect to $G$-Brownian motion.\\

\begin{defi}\quad
For a fixed $\alpha ,$ the space of rough paths
$\mathscr{C}^\alpha([0,T],\R^d)$ on $[0,T]$ consists of pairs
$(X,\X)$ satisfying ``Chen's identity''\eqref{chen} and the
condition of finite $\alpha$-H\"{o}lder norm and
$2\alpha$-H\"{o}lder norm respectively for $X$ and $\X.$ For any
$\BX:=(X,\X) \in \mathscr{C}^\alpha([0,T],\R^d),$ define its
semi-norm as the following
$$
\|\BX \|_{\FC^{\alpha}}:= \| X\|_{\alpha}+(\| \X \|_{2\alpha}
)^{\frac12}.
$$

\end{defi}

\begin{defi}\quad
A path $Y \in \oc^\alpha([0,T],\R^m)$ is said to be controlled by a
given path $X \in \oc^\alpha([0,T],\R^d),$ if there exists $Y' \in
\oc^\alpha([0,T],\mathcal {L}(\R^d,\R^m)),$ such that the remainder
term
$$
R_{s,t}^Y:=Y_{s,t}-Y'_s X_{s,t},
$$
satisfies $\|R^Y\|_{2\alpha} < \infty.$

\end{defi}

Here $\mathcal {L}(\R^d,\R^m)$ means the space of linear functions
from $\R^d$ to $\R^m,$ which is indeed $\R^{dm}.$ Denote the
collection of controlled rough paths by $\mathcal
{D}^{2\alpha}_X([0,T],\R^m).$ In addition, $Y'$ is called the
$Gubinelli \ derivative$ of $Y.$ For $(Y,Y') \in \mathcal
{D}^{2\alpha}_X([0,T],\R^m),$ we define its semi-norm by
$\|Y,Y'\|_{X,2\alpha}:=\|Y'\|_\alpha + \|R^Y\|_{2\alpha}.$

For example, given any $F \in \oc_b^2(\R^d,\R^m),$ the set of
bounded functions from $\R^d$ to $\R^m$ with bounded derivatives up
to order 2, one can easily check that $(Y,Y'):=(F(X),DF(X)) \in
\mathcal {D}^{2\alpha}_X([0,T],\R^m).$ In general, $Y'$ is not
uniquely determined by $Y,$ especially when $X$ is rather smooth.
However, if the underlying path $X$ is truly rough, $Y'$ can be
uniquely decided by $Y$ (see \cite{FS13}, \cite{FH14} for details).

The next theorem for the definition of rough integral based on
controlled paths is obtained in \cite{G04}, also see \cite{FH14},
\cite{L98}, \cite{LQ02}.

\begin{thm}{(\bf\small\itshape{Gubinelli,Lyons})}\quad
Suppose $\BX \in \FC^\alpha ([0,T],\R^d),$ and $(Y,Y') \in \mathcal
{D}^{2\alpha}_X([0,T],\mathcal {L}(\R^d,\R^n)).$ Then the following
compensated Riemann sum converges.
\begin{equation}
\int_0^T Y d\BX := \lim_{|\mathcal {P}|\rightarrow 0} \sum_{(s,t)
\in \mathcal{P}} (Y_s X_{s,t}+Y'_s \X_{s,t}),
\end{equation}
where $\op$ are partitions of $[0,T],$ with modulus $|\op|
\rightarrow 0.$ Furthermore, one has the bound
\begin{equation}
|\int_s^t Y_r d\BX_r -Y_sX_{s,t}- Y'_s \X_{s,t}| \leq K
(\|X\|_\alpha \|R^Y\|_{2\alpha} + \|\X\|_{2\alpha} \| Y'
\|_{\alpha})|t-s|^{3\alpha},
\end{equation}
where $K$ is a constant depending only on $\alpha.$
\end{thm}

Here one should note that
$\mathcal{L}(\R^d,\mathcal{L}(\R^d,\R^n))\cong
\mathcal{L}(\R^d\otimes \R^d,\R^n),$ so $\int_0^T Y d\BX \in \R^n,$
and $m=dn$ where $m$ is in the definition of controlled paths.

The Norris lemma was first established in \cite{N86}, and is viewed
as a quantitative version of Doob-Meyer's decomposition. A
deterministic quantitative Norris Lemma is given in \cite{CHLT15}.
It means that a rough integral can be distinguished from a rather
``smooth'' integral, essentially by the uniqueness of Gubinelli's
derivative, when the given rough path is ``truly rough''. Precisely,
one has the following definition and theorem.

\begin{defi}\quad
A path $X \in \oc^\alpha([0,T],\R^d)$ is said to be
$\theta$-H\"{o}lder rough for some given $\theta \in (0,1),$ on the
scale of $\varepsilon_0 > 0,$ if there exists a constant $L>0,$ such
that for any $a \in \R^d, $ $s \in [0,T],$ and $\varepsilon \in
(0,\varepsilon_0],$ there always exists $t \in [0,T],$ satisfying
$$
 |t-s|<\vep, \ and \    |a \cdot X_{s,t}| \geq L \varepsilon^\theta |a|.
$$
The largest value of such $L$ is called the modulus of $\theta
$-H\"{o}lder roughness of $X,$ denoted by $L_\theta(X).$ It is
obvious that the modulus $L_\theta(X)$ has the following expression:
\begin{equation}
L_\theta(X) = \inf_{|a|=1,s \in [0,T],\varepsilon \in
(0,\varepsilon_0]} \sup_{|t-s| \leq \varepsilon}
\frac{1}{\varepsilon^\theta}|a\cdot X_{s,t}|.
\end{equation}
\end{defi}

\begin{thm}({\bf\small\itshape{Norris lemma for rough paths}})\quad\label{NRP}
Suppose $\BX=(X,\X)\in \FC^\alpha([0,T],\R^d),$ with $X$ $\theta
$-H\"{o}lder rough for some $\theta < 2\alpha.$ Given $(Y,Y') \in
\FD_X^{2\alpha}([0,T],\mathcal {L}(\R^d,\R^n))$ and $Z \in
\oc^\alpha([0,T],\R^n),$ set
\begin{equation}
I_t=\int_0^t Y_s d\BX_s + \int_0^t Z_s ds,
\end{equation}
and
\begin{equation}
\mathcal {R}= 1+L_\theta(X)^{-1} + \|\BX\|_{\FC^\alpha} +
\|Y,Y'\|_{X,2\alpha} +|Y_0| +|Y'_0| +\|Z\|_\alpha +|Z_0|.
\end{equation}
Then one has the bound
\begin{equation}
\|Y\|_\infty  +\|Z\|_\infty \leq M \oor^q  \|I\|_\infty^r,
\end{equation}
for some constant $M,q \text{ and }r,$ only depending on $\alpha,
\theta, T.$

\end{thm}


\subsection{The $G$-expectation theory}

To introduce $G$-expectation theory, firstly we need to give a short
description of the sublinear expectation. Let $\Omega$ be a given
set and $\mathcal {H}$ be a linear space of real valued functions on
$\Omega$ containing constants. Furthermore, suppose
$\varphi(X_1,...,X_n) \in \mathcal{H}$ if $X_1,...,X_n \in
\mathcal{H} $ for $\varphi \in \oc_{b.Lip}(\R^n),$ the space of
bounded Lipschitz functions. The space $\Omega$ is the sample space
and $\mathcal {H}$ is the space of random variables.

\begin{defi}\quad
A sublinear expectation $\E$ is a functional $\E: \mathcal{H}
\rightarrow \R$ satisfying:
\begin{itemize}

\item $\E[c]=c,\qquad \forall\;c\in\,\mathbb{R};$
\item  $\E[X_1] \geq \E[X_2] \qquad if\; X_1\geq X_2;$
\item  $\E[\lambda X]=\lambda\E[X],\qquad \lambda\geq0 \quad X\in
\mathcal{H};$
\item $\E[X+Y]\leq \E[X]+\E[Y], \qquad    X,Y \in
\mathcal{H}.$

\end{itemize}

The triple $(\Omega, \mathcal{H}, \E)$ is called a sublinear
expectation space.
\end{defi}

%

\begin{defi}\quad
In a sublinear expectation space $(\Omega,\mathcal{H} ,\E)$, a
random vector $Y=(Y_{1},\cdot \cdot \cdot,Y_{n})$, $Y_{i}\in
\mathcal{H}$, $i=1...n,$ is said to be independent of another random
vector $X=(X_{1},\cdot \cdot \cdot,X_{m})$, $X_{i}\in \mathcal{H}$
under $\E[\cdot]$, if for every test function $\varphi \in
\oc_{b.Lip}(\mathbb{R}^{m}\times \mathbb{R}^{n}),$ we have
$\E[\varphi(X,Y)]=\E[ \E [\varphi(x,Y)]_{x=X}]$.
\end{defi}

\begin{rem}\quad
If $Y$ is independent of $X,$ one fails to get that $X$ is
independent of $Y$ automatically. Indeed, this is a main difference
between $G$-expectation theory and the classical case. There are
nontrivial examples explaining this point. See Chapter 1 in
\cite{P10}.

\end{rem}

\begin{defi}\quad
Let $X_{1}$ and $X_{2}$ be two $n$-dimensional random vectors
defined respectively in sublinear expectation spaces $(\Omega_{1}
,\mathcal{H}_{1},\E_{1})$ and $(\Omega_{2},\mathcal{H}
_{2},\E_{2})$. They are called identically distributed, denoted by
$X_{1}\overset{d}{=}X_{2}$, if $\E_{1}[\varphi(X_{1}
)]=\E_{2}[\varphi(X_{2})]$, for all $  \varphi \in \oc_{b.Lip}
(\mathbb{R}^{n})$.
\end{defi}

\begin{defi}{(\bf\small\itshape{$G$-normal distribution})}\quad
A $d$-dimensional random vector $X=(X_{1},\cdot \cdot \cdot,X_{d})$
in a sublinear expectation space $(\Omega,\mathcal{H},\E)$ is called
$G$-normally distributed if $\E[|X|^3] < \infty$ and for each
$a,b\geq0$
\[
aX+b\bar{X}\overset{d}{=}\sqrt{a^{2}+b^{2}}X,
\]
where $\bar{X}$ is an independent copy of $X$, i.e.,
$\bar{X}\overset{d}{=}X,$ $\bar{X}\text{ independent of } X,$ and
\[
G(A):=\frac{1}{2}\E[X'AX ]:\mathbb{S}%
_{d}\rightarrow \mathbb{R},
\]
Here $\mathbb{S}_{d}$ denotes the collection of $d\times d$
symmetric matrices.
\end{defi}

By Theorem 1.6 in Chapter 3 of \cite{P10}, we know that if
$X=(X_1,\cdots, X_d )$ is $G$-normally distributed,
$u(t,x):=\E[\varphi(x+\sqrt{t}X)],$ $(t,x) \in [0,\infty)\times
\R^d,$ is the unique viscosity solution of the following $G$-heat
equation:
\begin{equation}\label{Gheat}
\partial_t u-G(D_x^2u)=0, \ u(0,x)=\varphi(x),
\end{equation}
with function $G$ defined as above. \\

 Conversely, fixed any monotonic, sublinear function
$G(\cdot):\mathbb{S} _{d} \rightarrow \mathbb{R},$ one could
construct the sublinear expectation space $(\Omega,\mathcal {H},
\E).$\\

Now let $\Omega= \oc^0(\R^+,\R^d),$ the space of $\R^d$ valued
continuous paths $(\omega)_{t\geq0}$ vanishing at the origin. Denote
the coordinate process by $B_t$ and $u^{\varphi(\cdot)}(t,x)$ the
unique viscosity solution to the $G$-heat equation \eqref{Gheat}
with initial function $\varphi.$ Define $L_{ip}(\Omega_T):=\{
\varphi(B_{t_1\wedge T},...,B_{t_k\wedge T}): k\in \mathbb{N},
t_1,...t_k \in [0,\infty), \varphi\in \oc_{b.Lip}(\R^{k \times d})
\}$ for any $T>0$ and $L_{ip}(\Omega):= \bigcup_{n=1}^\infty
L_{ip}(\Omega_n) .$ We define a mapping $\E$ from $L_{ip}(\Omega)$
to $\R$ by recursively solving the $G$-heat equation:
\begin{equation}
\E[\varphi(B_{t_1}, B_{t_2}-B_{t_1},...,B_{t_n}-B_{t_{n-1}})]:=
\E[\varphi^{t_n-t_{n-1}}(B_{t_1},
B_{t_2}-B_{t_1},...,B_{t_{n-1}}-B_{t_{n-2}})],
\end{equation}
where
$\varphi^{t_n-t_{n-1}}(x_1,...,x_{n-1}):=u^{\varphi(x_1,...,x_{n-1},\cdot)}(t_n-t_{n-1},0).$
One can check that $\E[\cdot]$ is well defined and it is a sublinear
expectation on $L_{ip}(\Omega).$ Furthermore, one could define the
time consistent conditional expectation $\E[\cdot | \Omega_s]$ as
the mapping from $L_{ip}(\Omega)$ to $L_{ip}(\Omega_s)$ by
\begin{equation}
\E[\varphi(B_{t_1},
B_{t_2}-B_{t_1},...,B_{t_n}-B_{t_{n-1}})|\Omega_{s}]=\psi(B_{t_1},
B_{t_2}-B_{t_1},...,B_{s}-B_{t_{i-1}}), \ \ s\in [t_{i-1},t_i),
\end{equation}
with
$\psi(x_1,...x_i)=\E[\varphi(x_1,...,x_i+B_{t_i}-B_{s},...,B_{t_n}-B_{t_{n-1}})].$
Here is a collection of properties for this mapping.
\begin{itemize}
\item $\E[\xi|\Omega_t]= \xi,$ for any $\xi \in
L_{ip}(\Omega_t).$
\item $\E[X+Y|\Omega_t] \leq \E[X|\Omega_t]+ \E[Y|\Omega_t].$
\item $\E[\xi X| \Omega_t]= \xi^{+} \E[X|\Omega_t] + \xi^- \E[-X|
\Omega_t],$ for any $\xi \in L_{ip}(\Omega_t)$
\item $\E[\E[X|\Omega_t]|\Omega_s] = \E[X| \Omega_{t\wedge s}],$
in particular, $\E[\E[X|\Omega_t]]=\E[X].$
\item $\E[X|\Omega_t]=\E[X] ,$ if $X$ is independent of $L_{ip}(\Omega_t).$
\item $\E[X+\xi| \Omega_t] = \E[X| \Omega_t] + \xi,$ for any $\xi \in L_{ip}(\Omega_t), X \in L_{ip}(\Omega).$

\end{itemize}

From now on, we suppose the function $G$ non-degenerate, i.e., there
exists two constants $0<\underline{\sigma}^2 \leq \bar{\sigma}^2 <
\infty,$ such that
$$
\frac12 \underline{\sigma}^2 tr(A-B) \leq G(A)-G(B) \leq \frac12
\bar{\sigma}^2 tr(A-B).
$$
In the case that $\bar{\sigma}=\underline{\sigma},$ the function $G$
is linear, so $G$-framework is the
classical Wiener case.\\

For each $p \geq 1,$ $L_G^p(\Omega_T)$ denotes the completion of the
linear space $L_{ip}(\Omega_T),$ under norm $\| \cdot
\|_{L_G^p}:=\{\E[|\cdot|^p]\}^{\frac{1}{p}}.$ Obviously, for any
$p\leq q,$ $L_G^q \subseteq L_G^p.$ Furthermore, the conditional
expectation $\E[\cdot|\Omega_t]$ could be continuously extended to a
mapping from $L_G^1(\Omega)$ to $L_G^1(\Omega_t)$ and the extended
mapping adopts the above properties.\\

To give a description of elements in $L_G^p,$ Denis, Hu and Peng
gave the following representation of $\E[\cdot]$ by stochastic
control methods in \cite{DHP11}. Also see Hu and Peng \cite{HP09}
for an intrinsic and probabilistic method.

\begin{thm}\label{Grep}\quad
Assume $\Gamma $ is a bounded, convex and closed subset of
$\R^{d\times d},$ which represents function $G,$ i.e.,
$$
G(A)=\frac12 \sup_{\gamma \in \Gamma} tr(A\gamma \gamma'), for\ A
\in \mathbb{S}_d.
$$
Denote the Wiener measure by $P^0.$ Then, for any time sequence
$0=t_0 < t_1...<t_k ,$ the $G$-expectation has the following
representation

\begin{eqnarray*}
\E[\varphi(B_{t_0,t_1},...,B_{t_{k-1},t_k})] &=& \sup_{a \in
\mathcal {A}^{\Gamma} } E_{P^0}[\varphi(\int_0^{t_1}a_s
dB_s,...,\int_{t_{k-1}}^{t_{k}}a_s dB_s )]  \\
&=& \sup_{P^{a} \in \mathscr{P}^{\Gamma} }
E_{P^a}[\varphi(B_{t_0,t_1},...,B_{t_{k-1},t_k})] ,
\end{eqnarray*}

where $\mathcal {A}^{\Gamma}$ is the set of progressively measurable
processes with values in $\Gamma$ and $\mathscr{P}^\Gamma$ is the
set of laws of $\int_0^. a_s dB_s$ with $a \in \mathcal
{A}^{\Gamma}$ under Wiener measure. Furthermore,
$\mathscr{P}^{\Gamma}$ is tight.

\end{thm}

According to this theorem, one could extend $\E$ from $L_G^p$ to any
Borel measurable random variable by defining
$$\|\cdot\|_{\mathbb{L}^p}:=\sup_{P^{a} \in
\mathscr{P}^{\Gamma} } E_{P^a}^{\frac{1}{p}}[|\cdot|^p].$$ It is
simple to check that if $X \in L_G^p,$ then $\| X
\|_{\mathbb{L}^p}=\| X \|_{L_G^p}.$\\

Next, we introduce the capacity corresponding to the $G$-expectation
and give the description of $L_G^p.$ Define
$$
\hc(A):= \sup_{P \in \mathscr{P}^{\Gamma} }P(A), for \ A \in
\mathscr{B}(\Omega_T).
$$

\begin{defi}\quad
A property is said to hold ``quasi-surely''(q.s.) with respect to
$\hc$, if it holds true outside a $\hc$-polar set (Borel set with
capacity 0), and is denoted by $\hc-q.s..$\\

\end{defi}

\begin{defi}\quad
A process $Y$ on $[0,T]$ is said to be a quasi-surely modification
of another process $X,$ if for any $t \in [0,T]$
$$
Y_t=X_t, \ \ \ \hc-q.s..
$$

\end{defi}

If a property stands true $ \hc-q.s.,$ then for any $P \in
\mathscr{P}^{\Gamma},$ it holds true $P-a.s..$ By the definition of
$L_G^p,$ we do not distinguish two random variables if they are
equal outside a polar set.

\begin{defi}\quad
Equip the space $\Omega_T$ with the uniform topology. A mapping $X$
on $\Omega_T$ with values in $\R $ is said to be quasi-continuous if
for any $\varepsilon >0,$ there exists an open set $O,$ with
$\hc(O)< \vep$ such that $X$ is continuous in $O^c.$
\end{defi}

\begin{defi}
One says that $X:\Omega_T \rightarrow \R$ has a quasi-continuous
version if there exists a quasi-continuous function $Y,$ such that
$X=Y,  \ \hc-q.s..$

\end{defi}

\begin{thm}\quad
One has the following representation for $L_G^1,$
$$
L_G^1(\Omega_T)=\{ X \in \mathscr{B}(\Omega_T) : X \text{ has a
quasi-continuous version, } \lim_{n\rightarrow \infty} \|
|X|1_{\{|X|>n\} } \|_{\mathbb{L}^1}=0 \}.
$$
\end{thm}

\begin{prop}\label{gconverge}\quad
Assume that $(X_n)_{n\geq 1}$ is a sequence of random variables, and
converges to $X$ in the sense of $\|\cdot \|_{\mathbb{L}^p}$. Then
the convergence holds in the sense of capacity, i.e., for any $\vep
>0,$
$$
\hc(|X_n-X| > \vep)\stackrel{n}{\rightarrow} 0.
$$
Furthermore, there exists a subsequence $(X_{n_k})_{k \geq 1}$
converging to $X$ quasi-surely.
\end{prop}

\begin{rem}\quad
It is vital to point out that though the above proposition holds
true in the $G$-framework, even sup linear expectation framework,
the dominated convergence theorem (the quasi-surely version), and
the claim that quasi-surely convergence implies convergence in
capacity,
all fail in $G$-framework. \\

\end{rem}

Now we introduce the stochastic integral (It\^{o}'s integral)
for one-dimension case in $G$-framework.\\

Denote $M_G^{p,0}(0,T)$ the collection of processes with form
$$
\eta_t(\omega)=\sum_{i=0}^{N-1} \xi_i(\omega)1_{[t_i,t_{i+1})}(t),
$$
for a partition $\{0=t_0 <...<t_N=T\}$ and $\xi_i \in
L_{ip}(\Omega_{t_i}),i=0...N-1.$ Then denote by $M_G^{p}(0,T)$ the
completion of $M_G^{p,0}(0,T)$ under norm
$\|\cdot\|_{M_G^p}:=\{\E\int_0^T |\eta_s|^p ds\}^{\frac{1}{p}}.$

\begin{defi}\quad
For each $\eta \in M_G^{2,0}(0,T)$, one has the mapping $I$ from
$M_G^{2,0}(0,T)$ to $L_G^2(\Omega_T):$
\begin{equation}
I(\eta)=\int_0^T \eta_s dB_s := \sum_{i=0}^{N-1}
\xi_i(B_{t_{i+1}}-B_{t_i}).
\end{equation}
\end{defi}

It has been shown (see \cite{P07a},\cite{P08a},\cite{P10}) that the
mapping is continuous and can be extended to the completion space
$M_G^2(0,T).$  Define the quadratic variation processes $\langle B
\rangle $ of $G$-Brownian motion by
\begin{equation}
\langle B \rangle_t := B_t^2 - 2 \int_0^t B_s dB_s.
\end{equation}
It can be shown that $\underline{\sigma}^2 \leq \frac{d\langle B
\rangle_t}{dt} \leq \bar{\sigma}^2, \ \hc-q.s.,$ where
$\underline{\sigma}=\sqrt{-\E[-B_1^2]}$ and
$\bar{\sigma}=\sqrt{\E[B_1^2]}.$ In $G$-expectation theory, $\langle
B \rangle$ shares properties of independent stationary increment
just as $G$-Brownian motion. Moreover, the following integral of a
process in $M_G^{1,0}(0,T)$ can be continuously extended to the
completion $M_G^1(0,T).$
\begin{equation}
\int_0^T \eta_t d\langle B \rangle_t := \sum_{i=0}^{N-1}
\xi_i(\langle B \rangle_{t_{i+1}}-\langle B\rangle_{t_i}):
M_G^{1,0}(0,T) \rightarrow L_G^1(\Omega_T),
\end{equation}
where $\eta$ is defined as above, only $L_G^2$ replaced by $L_G^1.$\\

For the multi-dimensional case, one could obtain similar results.
Indeed, let $(B_t)_{t\geq 0}$ be a d-dimensional $G$-Brownian
motion. For any $a \in \R^d, $ $B^a:= a \cdot B$ is still a
$G_a$-Brownian motion. Then according to results in one-dimensional
case, one could define integrals with respect to $B^a, \langle B^a
\rangle,$ and obtain continuity for these mappings. Furthermore, the
mutual variation process $\langle B^a, B^{\bar{a}} \rangle_t$ could
be
defined by polarization.\\

%

At last, we end this subsection with It\^{o}'s formula in
$G$-framework. The proof could also be obtained in \cite{P10}.

\begin{thm}\label{GIF}\quad
Let $\Phi$ be a twice continuously differentiable function on $\R^n$
with polynomial growth for the first and second order derivatives.
Suppose $X$ is a It\^{o} process, i.e.
$$
X_t^\nu=X_0^\nu+\int_0^t \alpha_s^\nu ds +\int_0^t \eta_s^{\nu ij}
d\langle B^i,B^j \rangle_+ \int_0^t \beta _s^{\nu j }dB_s^j
$$
where $\nu=1,...,n,\ i,j=1,...,d,$ $\alpha_s^\nu,\eta_s^{\nu
ij},\beta _s^{\nu j }$ are bounded processes in $M_G^2(0,T).$ Here
and from now on, repeated indices means summation over the same
ones. Then for each $t \geq s \geq 0$ we have in $L_G^2(\Omega_t):$

\begin{eqnarray*}
\Phi(X_t)-\Phi(X_s)&=&\int_s^t \partial_{x^{\nu}} \Phi(X_u) \beta
_u^{\nu j } dB_u^j + \int_s^t \partial_{x^{\nu}} \Phi(X_u)
\alpha_u^{\nu} du\\
                   &+&\int_s^t[\partial_{x^{\nu}}\Phi(X_u)\eta_u^{\nu
ij}+\frac12 \partial^2_{x^{\mu} x^{\nu}}\Phi(X_u) \beta _u^{\mu i
}\beta _u^{\nu j }]d \langle B^i, B^j \rangle_u.
\end{eqnarray*}
\end{thm}

\section{$G$-Stochastic Integral as Rough Integral}

Firstly we give the $G$-expectation version of Kolmogorov criterion
for rough paths, the proof of which is adapted from the classical
case (see Theorem 3.1 in \cite{FH14}).

\begin{thm}\label{Kol}\quad
For fixed $q\geq 2, \beta
> \frac{1}{q},$ assume $X(\omega): [0,T]\rightarrow \R^d$ and $\X(\omega): [0,T]^2
\rightarrow \R^d\otimes \R^d$ are processes with $X_t \in
L_G^q(\Omega_T), \X_{s,t} \in L_G^{\frac{q}{2}}(\Omega_T),\forall
s,t \in [0,T],$ and satisfy relation $(\ref{chen}),$ $\hc-q.s.$. If
for any $s,t \in [0,T],$ one has bounds
\begin{equation}
\|X_{s,t}\|_{L_G^q} \leq C|t-s|^{\beta}, \;
\|\X_{s,t}\|_{L_G^{\frac{q}{2}}} \leq C |t-s|^{2\beta},
\end{equation}
for some constant $C$. Then for all $\alpha \in [0, \beta-
\frac{1}{q}),$ $(X,\X)$ has a $\hc-q.s.$ continuous modification,
and there exist $K_{\alpha}\in L_G^q, \K_{\alpha} \in
L_G^{\frac{q}{2}} $ such that for any $s,t \in [0,T],$ one has
inequalities
\begin{equation}
|X_{s,t}| \leq K_{\alpha} |t-s|^{\alpha}, \  |\X_{s,t}| \leq
\K_{\alpha} |t-s|^{2\alpha},  \ \  \  \hc-q.s..
\end{equation}
In particular, if $\beta - \frac{1}{q} > \frac13,$ then $\hc-q.s.$
$\BX=(X,\X)$ belongs to $\FC^{\alpha}([0,T],\R^d),$ for any $\alpha
\in (\frac13,\beta - \frac{1}{q}).$
\end{thm}

\begin{proof}
Without loss of generality suppose T=1, and define dyadic partition
as $D_n=\{i 2^{-n}, i=0... 2^n\}.$ Set
$$
K_n= \sup_{t \in D_n}|X_{t, t+2^{-n}}|, \; \K_n = \sup_{t \in D_n}
|\X_{t,t+2^{-n}}|.
$$
Note that since $D_n$ are finite sets, $K_n, \K_n $ belong to $
L_G^q$ and $L_G^{\frac{q}{2}} $ respectively. Furthermore, one has
bounds
\begin{eqnarray*}
\E(K_n^q) \leq \sum_{D_n} \E|X_{t,t+2^{-n}}|^q \leq C^q
(\frac{1}{2^n})^{\beta q -1}\\
\E(\K_n^{\frac{q}{2}}) \leq \sum_{D_n} \E|\X_{t,
t+2^{-n}}|^{\frac{q}{2}} \leq C^{\frac{q}{2}} (\frac{1}{2^n})^{\beta
q -1}
\end{eqnarray*}
For any $s, t \in \bigcup_n D_n,$ there exists $m$ such that
$2^{-m-1} < t-s \leq 2^{-m},$ and a partition, $s=\tau_0< \tau_1
<...<\tau_N=t,$ with $(\tau_i, \tau_{i+1}) \in D_k,$ for some $k\geq
m+1.$ Also, we can choose such a partition that at most two
intervals in this partition are taken from the same $D_k$ for any
fixed $k \geq
m+1.$\\
Then one obtains
$$
|X_{s,t}| \leq \max_{0 \leq i \leq N}|X_{s,\tau_i}| \leq \sum_{i=0}
^{N-1} |X_{\tau_i, \tau_{i+1}}| \leq 2 \sum_{n \geq m+1} K_n.
$$
It follows that
$$
\frac{|X_{s,t}|}{|t-s|^{\alpha}} \leq 2\sum_{n \geq
m+1}(2^n)^{\alpha}K_n \leq K_{\alpha},
$$
where $K_{\alpha}:= 2\sum_{n \geq 0}2^{n\alpha} K_n.$ We can easily
check that $K_{\alpha} \in L_G^q$ since $K_n \in L_G^q.$\\
For the second order part $\X,$ by ``Chen's identity'', one has the
following inequalities,
\begin{eqnarray*}
|\X_{s,t}|&=&|\sum_{i=0}^{N-1} (\X_{\tau_i, \tau_{i+1}} + X_{s,
\tau_i}\otimes  X_{\tau_i,\tau_{i+1}}) | \\
&\leq& 2 \sum_{n\geq m+1} \K_n + \max_{0 \leq i \leq N}|X_{s,
\tau_{i+1}}|
\sum_{j=0} ^{N-1} |X_{\tau_j, \tau_{j+1}}|\\
 &\leq& 2 \sum_{n \geq
m+1} \K_n + (2 \sum_{n \geq m+1} K_n)^2.
\end{eqnarray*}
Then one obtains

$$
\frac{|\X_{s,t}|}{|t-s|^{2\alpha}}\leq \sum_{n\geq 1}2 \K_n
2^{2n\alpha} + K_\alpha^2,
$$
the right side of which can be checked to belong to
$L_G^{\frac{q}{2}}$.
\end{proof}

\subsection{$G$-It\^{o} integral as rough integral}

Let us consider the $G$-Brownian motion as rough paths. Suppose
$B=(B^{(1)},...,B^{(d)})$ is a $d$-dimensional $G-$Brownian motion
and
$$
\E[|B^{(i)}B^{(j)}|] = \bar{\sigma}^2_{ij},\ \
-\E[-|B^{(i)}B^{(j)}|] = \underline{\sigma}^2_{ij}.
$$
For simplicity, suppose for some positive number $\bar{\sigma},$
$\underline{\sigma}_{ij} \leq \bar{\sigma}_{ij}<\bar{\sigma}$ for
any $i,j.$ Firstly, it is obvious that the lifted $G$-Brownian
motion, $(B,\B):=(B, \int_s^t B_{s,r} d
B_r)=(\{B^{(i)}\}_{i=1}^d,\{\int_s^t B_{s,r}^{(i)} d
B_r^{(j)}\}_{1\leq i,j\leq d} )$ satisfies \eqref{chen}.
%
There remains the analytic condition to be checked. With an
application of Theorem \ref{Kol}, the following proposition would
stand for our claim that the lifted $G$-Brownian motion belongs to
the rough path space $\FC^{\alpha} ,$ $\hc-q.s..$

\begin{prop}\quad
One has the following inequalities
$$
\|B_{s,t}\|_{L_G^q} \leq C_{q,\bar{\sigma}}|t-s|^{\frac12}, \;
\|\B_{s,t}\|_{L_G^{\frac{q}{2}}} \leq C_{q,\bar{\sigma}}
|t-s|,\text{ for any } q \geq 2,
$$
 where $C_{q,\bar{\sigma}}$ is a constant
depending on $ k $ and $\bar{\sigma}.$
\end{prop}

\begin{proof}
It is obvious that $\|B_{s,t}\|_{L_G^q} \leq C_q |t-s|^{\frac 12}.$
Thanks to the property of stationary and independent increment for
$G$-Brownian motion, only $\E|\int_0^t B_r dB_r|^{2k} \leq C_k
t^{2k},$ for any $k \geq 1,$ left to be checked.
\\

Note that $\int_0^t B_r^{(i)} dB_r^{(j)}$ is a square integrable
continuous martingale under each $P \in \mathscr{P}^\Gamma$ by
Theorem \ref{Grep}. A combination of Burkholder-Davis-Gundy
inequality and Jenson's inequality tells that
\begin{eqnarray*}
\E|\int_0^t B_r^{(i)} dB_r^{(j)}|^{2k} &\leq& C_k \E |\int_0^t
(B_r^{(i)})^2 d \langle B^{(j)}
\rangle_r|^k \leq C_{k,\bar{\sigma}} \E|\int_0^t (B_r^{(i)})^2 dr|^k\\
                         & \leq & C_{k,\bar{\sigma}} t^k \E(\frac{1}{t}\int_0^t
                         |B_r^{(i)}|^{2k}dr)\\
                         &\leq &  C_{k,\bar{\sigma}} t^{2k},
                         \end{eqnarray*}
where $C_{k,\bar{\sigma}}$ is a constant depending on $ k $ and
$\bar{\sigma},$ which implies the result by basic inequalities.
\end{proof}

Since $(B,\B)$ are rough paths $\hc-q.s.$, for $(Y,Y')\in
\FD^{2\alpha}_{B(\omega)} \bigcap M_G^2,$ we denote $\int Y_r d
\BB_r$ as the rough integral and $\int Y_r d B_r$ as the It\^{o}
integral with respect to $G$-Brownian motion.

\begin{prop}({\bf\small\itshape{$G$-It\^{o} stochastic integral as rough
integral}})\label{itor}\quad Assume $(Y,Y')(\omega) \in
\FD^{2\alpha}_{B(\omega)}([0,T],\mathcal{L}(\R^d,\R^n)),$ $\hc-q.s,$
and $Y,Y' \in M_G^2(0,T),$ with $Y_t,Y'_t$ in $L_G^2(\Omega_t),$ for
any $t \in [0,T].$ Furthermore, suppose
$\|\|Y\|_\alpha\|_{\mathbb{L}^2},\| \|Y'\|_\alpha \|_{\mathbb{L}^2}
< \infty.$ Then the identity holds,
\begin{equation}
\int_0^T Y_r d\BB_r = \int_0^T Y_r dB_r, \ \ \hc-q.s..
\end{equation}
 In particular, $\sum_{(u,v) \in \op} (Y_u B_{u,v}
+ Y'_u \B_{u,v})$ converges to $\int_0^T Y_r d\BB_r$ in the
$L_G^2$-norm sense.

\end{prop}

\begin{proof}

Suppose $\op$ is any partition of $[0,T]$ and set
$Y_t^\op:=\sum_{[u,v] \in \op}Y_u 1_{[u,v)}(t).$ Then we have
inequalities,
\begin{eqnarray*}
\E|\int_0^T (Y_t-Y_t^\op) d B_t|^2 &\leq & C_{\bar{\sigma}} \E
\int_0^T |Y_t-Y_t^\op|^2 dt \leq C_{\bar{\sigma}} \sum_{\op}
\int_u^v (t-u)^{2\alpha}
\E\|Y\|_\alpha ^2 dt\\
                                  & \leq & C_{\bar{\sigma}} T |\op|^{2\alpha}
                                  \E\|Y\|_\alpha ^2.
\end{eqnarray*}
In particular, $\sum_{\op} Y_u B_{u,v} \stackrel{|\op|\rightarrow
0}{\rightarrow} \int_0^T Y_t dB_t,$ in the sense of $L_G^2$-norm, so
according to Proposition \ref{gconverge}, there exists a
subsequence, denoted as $\sum_{\op_n}Y_u B_{u,v},$ converging to
$\int_0^T Y_r dB_r, \ \ \hc-q.s.$. \\

By the definition of rough integral, $\sum_{\op_n}(Y_u B_{u,v}+ Y'_u
\B_{u,v}) \rightarrow \int_0^T Y_r d\BB_r,$ $\hc-q.s..$ We claim
that, as the difference term of the two sequences, $\sum_{\op_n}
Y'_u \B_{u,v}$ converges to $0$ in $L_G^2-$norm sense, and then
according to this, there exists a subsequence $\{\op_{n_{k}}\}$ such
that $\sum_{\op_{n_{k}}} Y'_u \B_{u,v}$ converges to $0,$
$\hc-q.s.$, which implies the desired result. At last, one has the
following inequalities,

\begin{eqnarray*}
&& \| \sum_{\op_{n_k}}Y'_u \B_{u,v} \|_{L_G^2}^2 =
\E[\sum_{\op_{n_k}}|Y'_u|^2 |\B_{u,v}|^2 ]\\
                                          &=& \E[|Y_{u_1}'|^2 |\B_{u_1,v_1}|^2+\cdots + |Y_{u_{l_k}}'|^2
                                          \E[|\B_{u_{l_k},v_{l_k}}|^2|\Omega_{u_{l_k}}]]\\
                                          & \leq & \E[|Y_{u_1}'|^2 |\B_{u_1,v_1}|^2+\cdots + |Y_{u_{l_k-1}}'|^2 |\B_{u_{l_k-1},v_{l_k-1}}|^2+
                                          C|Y_{u_{l_k}}'|^2(v_{l_k}-u_{l_k})^2]\\
                                          &= & \E[|Y_{u_1}'|^2 |\B_{u_1,v_1}|^2+\cdots +\E[ |Y_{u_{l_k-1}}'|^2
                                          |\B_{u_{l_k-1},v_{l_k-1}}|^2 +C|Y_{u_{l_k}}'|^2(v_{l_k}-u_{l_k})^2
                                          |\Omega_{u_{l_k-1}}]]\\
                                          &\leq & \E[|Y_{u_1}'|^2 |\B_{u_1,v_1}|^2+\cdots + \E[ |Y_{u_{l_k-1}}'|^2
                                          |\B_{u_{l_k-1},v_{l_k-1}}|^2|
                                          \Omega_{u_{l_k-1}}]
                                          + C \E[|Y_{u_{l_k}}'|^2(v_{l_k}-u_{l_k})^2 |
                                          \Omega_{u_{l_k-1}}]]\\
                                          &\leq & \E[|Y_{u_1}'|^2 |\B_{u_1,v_1}|^2+\cdots + |Y_{u_{l_k-2}}'|^2
                                          |\B_{u_{l_k-2},v_{l_k-2}}|^2
                                          + C(|Y_{u_{l_k-1}}'|^2 (v_{l_k-1}-u_{l_k-1})^2+|Y_{u_{l_k}}'|^2(v_{l_k}-u_{l_k})^2)]\\
                                          & \leq & \cdots \cdots\\
                                          & \leq & C\E[\sum_{\op_{n_k}} |Y'_u|^2
                                          (v-u)^2]\\
& \leq & C |\op_{n_k}| \|Y'\|_{M_G^2}^2.
\end{eqnarray*}
The last inequality follows from the convergence, $\sum_{(u,v) \in
\op} Y'_u 1_{[u,v)}(t) \stackrel{|\op|}{\rightarrow} Y'$ in the
sense of $M_G^2.$ Indeed,
\begin{eqnarray*}
\E \int_0^T |\sum_{(u,v) \in \op} Y'_u 1_{[u,v)}(t) - Y'_t |^2 dt
&\leq &
\sum_\op \E \int_u^v |Y'_u-Y'_t|^2 dt\\
&\leq & \E\|Y'\|_\alpha^2 T |\op|^{2\alpha} /(2\alpha+1) \rightarrow
0.
\end{eqnarray*}
\end{proof}

\begin{rem}\quad
Here $Y,Y'\in M_G^2$ means every element of $Y$ and $Y'$ belongs to
$M_G^2.$ According to the above proof, one can simply check that the
assumption $\| \|Y'\|_\alpha \|_{\mathbb{L}^2} < \infty$ could be
replaced by $|Y'|$ bounded.

\end{rem}

\begin{example}\label{ex1}\quad
$(i).$ For fixed $\alpha \in (\frac13,\frac12)$ and any function $F
\in \oc^2(\R,\R) $ with polynomial growth for the first and second
order derivatives, i.e.,
$$
|D^2F(x)|+|DF(x)| \leq C(1+|x|^k),
$$
for some positive constants $C,k,$ it is simple to check that
$(Y,Y'):= (F(B),DF(B)),$ where $B$ is a one-dimensional $G-$Brownian
motion, satisfies the assumption in the above proposition. Indeed,
according to Taylor's expansion,
\begin{eqnarray*}
F(B_t)-F(B_s) &=& DF(B_s+\lambda_1 (B_t-B_s))(B_t-B_s),\\
DF(B_t)-DF(B_s)&=& D^2F(B_s+\lambda_2 (B_t-B_s))(B_t-B_s),\\
F(B_t)-F(B_s)&=& DF(B_s )(B_t-B_s)+ \frac12 D^2F(B_s+\lambda_3
(B_t-B_s)) (B_t-B_s)^2,
\end{eqnarray*}
for some $\lambda_i(\omega) \in [0,1], i=1,2,3.$ \\

 By Theorem \ref{Kol}, it holds that,
\begin{eqnarray*}
\|F(B)\|_\alpha &\leq& \sup_{s,t\in [0,T] \atop \lambda_1 \in [0,1]
}
|DF(B_s+\lambda_1 (B_t-B_s) )| \|B\|_\alpha, \ \ \ \hc-q.s.;\\
\|DF(B)\|_\alpha &\leq& \sup_{s,t\in [0,T] \atop \lambda_2 \in [0,1]
}
|DF(B_s+\lambda_2 (B_t-B_s) )| \|B\|_\alpha, \ \ \ \hc-q.s.;\\
\|R^Y\|_{2\alpha} &\leq& \frac12 \sup_{s,t\in [0,T] \atop
\lambda_3\in[0,1]}|D^2F(B_s+\lambda_3(B_t-B_s) )| \|B\|_\alpha^2, \ \ \ \hc-q.s.,\\
\end{eqnarray*}
so $(F(B),DF(B)) \in \FD^{2\alpha}_{B(\omega)}([0,T],\R),\hc-q.s..$
Furthermore, by the polynomial growth condition and Theorem
\ref{Kol}, one can simply check that
$$
\|\|Y\|_\alpha\|_{\mathbb{L}^2},\| \|Y'\|_\alpha \|_{\mathbb{L}^2} <
\infty.
$$

$(ii).$ For a given function $f \in \oc^1(\R,\R),$ which satisfies
$$
|f(x)|+|Df(x)| \leq K(1+|x|^d),
$$
for some positive constants $K,d,$ define $(Z,Z'):=
(\int_0^.f(B_r)dB_r ,f(B_.)).$ Firstly, we need to show $Z \in
M_G^2(0,T).$ Define $Z_t^N:= \sum_{i=0}^{k_N} Z_{t_i}
1_{[t_i^N,t_{i+1}^N)}(t),$ where $\op^N:=\{0=t_0^N <
t_1^N<...<t_{k_N}^N=T \}$ is any sequence of partition with modulus
$|\op^N|$ converging to $0,$ and then one could obtain $Z^N
\stackrel{N}{\rightarrow} Z $ under the norm of $M_G^2.$ Indeed,
\begin{eqnarray*}
\E \int_0^T (Z_t^N-Z_t)^2 dt &\leq & \sum_{\op^N}
\int_{t_i^N}^{t_{i+1}^N} \E(Z_t^N-Z_t)^2 dt\\
                             &= & \sum_{\op^N}
\int_{t_i^N}^{t_{i+1}^N} \E (\int_{t_i^N}^t f(B_r) dB_r)^2 dt\\
                             &\leq & C_{\bar{\sigma}, K,d,}T |\op^N|
                             \rightarrow 0,
\end{eqnarray*}
where $C_{\bar{\sigma}, K,d,}$ is a constant depending on
$\bar{\sigma}, K,d.$ Secondly, one needs to check that
$(Z,Z')\in\FD^{2\alpha}_{B(\omega)}([0,T],\R),\hc-q.s..$ According
to Theorem \ref{Kol}, it is simple to obtain that $Z \in \oc^\alpha,
\hc-q.s.,$ and $\|\|Z\|_\alpha\|_{\mathbb{L}^q} < \infty,$ for any
$q \geq 2$ and $\alpha \in (\frac13, \frac12- \frac{1}{q}).$
Finally, only $R^Z \in \oc^{2\alpha}$ needs to be checked. Define
$H(x):=\int_0^x f(y) dy.$ Then $DH(x)=f(x), $ and $H(x) $ has
polynomial growth for the first and second derivatives. By
$G$-It\^{o}'s formula,
\begin{eqnarray*}
H(B_t)-H(B_s)= Z_{s,t} + \frac12 \int_s^t D f(B_r) d\langle B
\rangle_r, \ \ \ \hc-q.s..
\end{eqnarray*}
According to example $(i)$, $R^{H(B)}_{s,t}:=
H(B_t)-H(B_s)-f(B_s)B_{s,t}$ quasi-surely belongs to
$\oc^{2\alpha}.$ Since $\langle B \rangle_.$ is absolutely
continuous, one could say $R^Z_{s,t}:=Z_{s,t}-f(B_s)B_{s,t} \in
\oc^{2\alpha}, \hc-q.s..$

\end{example}

%

\begin{rem}\quad
It is easy to see that one could replace $B$ with It$\hat{o}$
processes and apply similar tricks for more examples.

\end{rem}


\subsection{$G$-Stratonovich integral as rough integral}

Firstly, we provide a description of Stratonovich integral with
respect to $G$-Brownian motion. Define $\langle Y,B^{(k)}
\rangle_t:= \lim_{|\op| \rightarrow 0} \sum_{(u,v) \in
\op}Y_{u,v}B_{u,v}^{(k)},$ for any $k=1,...,d,$ whenever the limit
exists in $L_G^1(\Omega_t),$ for any $t \in [0,T].$

\begin{prop}\label{b condi}\quad
For any $\beta=(\beta^{(1)},...,\beta^{(d)}) \in M_G^2,$ define
$Y_t:=\int_0^t \beta_r^{(l)} dB_r^{(l)}.$ Then one has
\begin{equation}
\langle Y, B^{(k)} \rangle_t = \int_0^t \beta_r^{(l)} d\langle
B^{(l)},B^{(k)} \rangle_r, \ \ \hc-q.s..
\end{equation}

\end{prop}

\begin{proof}
By linearity one only needs to show the case that $\beta$ is
one-dimensional, i.e. $Y_t=\int_0^t\beta_r dB_r^{(l)},$ for any
fixed
$l=1,...,d.$\\

 Step1: Suppose that $\beta_s \in M_G^{2,0},$ with the form
$\beta_s=\sum_{i=0}^{N-1} \xi_i 1_{[t_i,t_{i+1})}(s),$ $|\xi_i|\leq
K,i=0...N-1,$ and the partition $\mathcal {Q} := \{0=t_0<t_1 < t_2
<...<t_N=t  \}$ fixed. \\

For any partition $\op=\{0=\tau_0<\tau_1 < \tau_2 <...<\tau_M=t \},$
satisfying $|\op| \leq |\oq|,$ it holds that
\begin{eqnarray*}
&&\sum_\op
(Y_{\tau_i,\tau_{i+1}}B^{(k)}_{\tau_i,\tau_{i+1}})=\sum_\op
(\int_{\tau_i}^{\tau_{i+1}}\beta_s dB^{(l)}_s B^{(k)}_{\tau_i,\tau_{i+1}})\\
&=& \sum_{[\tau_i,\tau_{i+1}) \subset [t_j,t_{j+1})} (\xi_j
B^{(l)}_{\tau_i,\tau_{i+1}})B^{(k)}_{\tau_i,\tau_{i+1}}\\
&+ & \sum_{t_j \in
[\tau_i,\tau_{i+1})}(\xi_{j-1}(B_{t_j}^{(l)}-B^{(l)}_{\tau_i})  +
\xi_j(B^{(l)}_{\tau_{i+1}}-B^{(l)}_{t_j})) B^{(k)}_{\tau_i, \tau_{i+1}}\\
&=& \sum_{[\tau_i,\tau_{i+1}) \subset [t_j,t_{j+1})} (\xi_j
B^{(l)}_{\tau_i,\tau_{i+1}})B^{(k)}_{\tau_i,\tau_{i+1}}\\
&+ & \sum_{t_j \in
[\tau_i,\tau_{i+1})}(\xi_{j-1}(B_{t_j}^{(l)}-B^{(l)}_{\tau_i})  +
\xi_j(B^{(l)}_{\tau_{i+1}}-B^{(l)}_{t_j})) (B^{(k)}_{\tau_i, t_j}+B^{(k)}_{ t_j,\tau_{i+1}})\\
&=& \sum_{\op \vee \oq} (\xi_u B^{(l)}_{u,v})B_{u,v}^{(k)} +
\sum_{j=1}^{N-1}\xi_{j-1} B^{(l)}_{\tau_i,
t_j}B^{(k)}_{t_j,\tau_{i+1}} +\sum_{j=1}^{N-1}\xi_{j}
B^{(k)}_{\tau_i, t_j}B^{(l)}_{t_j,\tau_{i+1}}, \ \ \hc-q.s.,
\end{eqnarray*}
in the last equation of which we patch the two partitions together. \\

 According to Chapter 3 Lemma 4.6 in \cite{P10}, it suffices to show the convergence
$$
\sum_{j=1}^{N-1}\xi_{j-1} B^{(l)}_{\tau_i,
t_j}B^{(k)}_{t_j,\tau_{i+1}} +\sum_{j=1}^{N-1}\xi_{j}
B^{(k)}_{\tau_i, t_j}B^{(l)}_{t_j,\tau_{i+1}} \stackrel{|\op|}{
\rightarrow} 0,
$$
in the sense of $L_G^1.$ Indeed,
\begin{eqnarray*}
\E[|\sum_{j=1}^{N-1}\xi_{j-1} B^{(l)}_{\tau_i,
t_j}B^{(k)}_{t_j,\tau_{i+1}} +\sum_{j=1}^{N-1}\xi_{j}
B^{(k)}_{\tau_i, t_j}B^{(l)}_{t_j,\tau_{i+1}}|] \leq   2(N-1) K
|\op| \rightarrow 0.
\end{eqnarray*}

Step2: For any $\beta\in M_G^2,$ assume $\{\beta_s^n\}_n \subset
M_G^{2,0},$ converges to $\beta$ in the sense of $M_G^2.$ One has
inequalities,

\begin{eqnarray}
& &\E[|\sum_{\op}\int_u^v \beta_s dB^{(l)}_s B^{(k)}_{u,v}- \int_0^t
\beta_s
d\langle B^{(l)},B^{(k)}\rangle_s|]  \nonumber \\
&\leq & \E[|\sum_{\op}(\int_u^v \beta_s dB^{(l)}_s
B^{(k)}_{u,v}-\int_u^v
\beta_s^n dB^{(l)}_s B^{(k)}_{u,v}+\int_u^v \beta_s^n dB^{(l)}_s B^{(k)}_{u,v} \nonumber  \\
&-&\int_u^v \beta_s^n d\langle B^{(l)},B^{(k)}\rangle_s + \int_u^v
\beta_s^n
d\langle B^{(l)},B^{(k)}\rangle_s - \int_u^v \beta_s d\langle B^{(l)},B^{(k)}\rangle_s)|]  \nonumber \\
&\leq &  \E[| \sum_{\op} \int_u^v \beta_s dB^{(l)}_s
B^{(k)}_{u,v}-\int_u^v
\beta_s^n dB^{(l)}_s B^{(k)}_{u,v}|] + \E[|\sum_\op \int_u^v \beta_s^n dB^{(l)}_s B^{(k)}_{u,v} \nonumber \\
&-&\int_u^v \beta_s^n d\langle B^{(l)},B^{(k)}\rangle_s|] +
\E[|\sum_{\op}\int_u^v (\beta_s^n - \beta_s)d\langle
B^{(l)},B^{(k)}\rangle_s |] \label{*}
\end{eqnarray}

The second term in (\ref{*}) converges to $0$ by Step1. According to
definitions, the third term also converges to $0$.\\

At last, for the first term, since the calculation is carried in
$L_G^2$ and $B$ is a martingale under each $P\in \mathscr{P},$ one
obtains that
\begin{eqnarray*}
&& \E[|\sum_\op \int_u^v (\beta_s-\beta_s^n)dB^{(l)}_s B^{(k)}_{u,v}|]\\
&\leq& \sup_{P \in \mathscr{P}} \sum_\op E_P[|\int_u^v
(\beta_s-\beta_s^n)dB^{(l)}_s B^{(k)}_{u,v}|]\\
&\leq& \sup_{P \in \mathscr{P}} \bar{\sigma}^2 \sum_\op
\{E_P[\int_u^v (\beta_s-\beta_s^n)^2ds]\}^\frac12 |v-u|^\frac12\\
&\leq& \bar{\sigma}^2 \sup_{P \in \mathscr{P}} T^\frac12 \{\sum_\op
E_P[\int_u^v (\beta_s-\beta_s^n)^2ds]\}^\frac12\\
&\leq& \bar{\sigma}^2 T^\frac12 \{\E[\int_0^T
(\beta_s-\beta_s^n)^2ds\}^\frac12 \rightarrow 0.
\end{eqnarray*}

\end{proof}

\begin{coro}\quad
For
\begin{equation}
Y_t=\xi + \int_0^t \beta_s^{(j)} dB^{(j)}_s+\int_0^t \alpha_s ds +
\int_0^t \gamma^{(j,l)}_s d \langle B^{(j)},B^{(l)} \rangle_s,
\label{G-ito}
\end{equation}
with $\beta \in M_G^2 $, and $\alpha, \gamma \in
M_G^{1+\delta}(0,T), $ for some $\delta > 0,$ one has the
expression,
$$
\langle Y,B^{(k)}\rangle _t = \int_0^t \beta^{(j)}_s d\langle
B^{(j)},B^{(k)} \rangle_s, \ \ \hc-q.s..
$$
\end{coro}

\begin{proof}
Without loss of generality, the proof can be done by showing
$$
\sum_{(u,v)\in \op}  \int_u^v \gamma_s d\langle B^{(j)} \rangle_s
B^{(k)}_{u,v} \stackrel{|\op|}{\rightarrow} 0,\; \; \sum_{(u,v)\in
\op} \int_u^v \alpha_s ds B^{(k)}_{u,v}
\stackrel{|\op|}{\rightarrow} 0,
$$
in the sense of $L_G^1.$ We only show the first convergence. Indeed,
by boundedness for $\frac{d\langle B^{(j)} \rangle_t}{dt},$ one has
the following inequalities
\begin{eqnarray*}
&&\E[|\sum_\op  \int_u^v \gamma_s d\langle B^{(j)} \rangle_s
B^{(k)}_{u,v}|] \leq
\bar{\sigma}^2 \sup_{P \in \mathscr{P}} \sum_\op E_P[|B^{(k)}_{u,v}| |\int_u^v \gamma_s d s|]\\
& \leq & \bar{\sigma}^2 \sup_{P \in \mathscr{P}} \sum_\op
[E_P|\int_u^v \gamma_s ds|^{1+\delta}]^\frac{1}{1+\delta}
[E_P|B^{(k)}_{u,v}|^\frac{1+\delta}{\delta}]^\frac{\delta}{1+\delta}\\
& \leq & \bar{\sigma}^3 \sup_{P \in \mathscr{P}} \sum_\op
[E_P\int_u^v |\gamma_s|^{1+\delta} ds
]^\frac{1}{1+\delta}|v-u|^{\frac{\delta}{1+\delta}}  |v-u|^\frac12\\
&\leq & \bar{\sigma}^3 \sup_{P \in \mathscr{P}}  |\op|^\frac12
[\sum_\op E_P \int_u^v |\gamma_s|^{1+\delta} ds]^\frac{1}{1+\delta}
T^{\frac{\delta}{1+\delta}}\\
&\leq & \bar{\sigma}^3 T^{\frac{\delta}{1+\delta}}
\|\gamma\|_{M_G^{1+\delta}}|\op|^\frac12 \rightarrow 0 .
\end{eqnarray*}

\end{proof}

\begin{defi}({\bf\small\itshape{Stratonovich integration with respect to $G$-Brownian motion}})\quad
Suppose $Y=(Y_1,...,Y_d) \in M_G^2(0,T),$ and $\langle
Y^{(i)},B^{(i)} \rangle$ exist for any $i$. The Stratonovich
integral of $Y$ against $B,$ with value in $L_G^1,$ is given by
identity:
\begin{equation}
\int_0^t Y_s^{(i)} \circ dB^{(i)}_s := \int_0^t Y^{(i)}_s dB^{(i)}_s
+\frac12 \langle Y^{(i)},B^{(i)} \rangle_t, \ \ \hc-q.s..
\end{equation}
\end{defi}

\begin{prop}\quad
Assume $Y$ defined as $Y_t:=\int_0^t \beta^{(j)}_s dB^{(j)}_s, $
with $\beta \in M_G^2.$ Then, for partitions $\op$ of $[0,t]$ with
$|\op|\rightarrow 0,$ it holds that
\begin{equation}
L_G^1-\lim_{|\op| \rightarrow 0} \sum_{(u,v)\in \op}
\frac{Y_u+Y_v}{2} B^{(k)}_{u,v}= \int_0^t Y_s \circ dB^{(k)}_s
\end{equation}
\end{prop}

\begin{proof}
Suppose $t=T$ here. According to the definition of $\langle
Y,B^{(k)} \rangle,$ it suffices to show the following convergence
under the case that $\beta$ is one-dimensional,
$$
\sum_\op Y_u B^{(k)}_{u,v} \stackrel{L_G^1}{\rightarrow} \int_0^T
Y_r dB^{(k)}_r.
$$\\

Step1. If $\beta_s=\sum_{i=0}^{N-1} \xi_i 1_{[t_i, t_{i+1})}(s),$
with $\oq:={0=t_0<t_1<...<t_N=T},$ a fixed partition, one has the
identity
\begin{eqnarray*}
&&Y_r= \int_0^r \beta_s dB^{(l)}_s \\
&=& \sum_{i=0}^{N-2} \xi_i B^{(l)}_{t_i,t_{i+1}} 1_{[t_{i+1},T)}(r)
+
\sum_{i=0}^{N-1} \xi_i B^{(l)}_{t_i,r} 1_{[t_i, t_{i+1})}(r) +\xi_{N-1} B^{(l)}_{t_{N-1},t_N} 1_{\{T\}}(r)\\
&=& \sum_{i=0}^{N-1}(\sum_{j=0}^{i-1} \xi_j
B^{(l)}_{t_j,t_{j+1}}-\xi_i B^{(l)}_{t_i}) 1_{[t_i,t_{i+1})}(r) +
B^{(l)}_r \sum_{i=0}^{N-1} \xi_i
1_{[t_i,t_{i+1})}(r)+ \xi_{N-1} B^{(l)}_{t_{N-1},t_N} 1_{\{T\}}(r)\\
&=&\sum_{i=0}^{N-1} \tilde{\xi}_i 1_{[t_i,t_{i+1})}(r)+ B^{(l)}_r
\sum_{i=0}^{N-1} \xi_i 1_{[t_i,t_{i+1})}(r)+ \xi_{N-1}
B^{(l)}_{t_{N-1},t_N} 1_{\{T\}}(r),
\end{eqnarray*}
where we denote $ \tilde{\xi}_i:= (\sum_{j=0}^{i-1} \xi_j
B^{(l)}_{t_j,t_{j+1}}-\xi_i B^{(l)}_{t_i}),$ and $\sum_{j=0}^{-1}
\xi_j
B^{(l)}_{t_j,t_{j+1}}=0.$\\

It follows that
\begin{eqnarray}
&&\int_0^T Y_r dB^{(k)}_r = \sum_{i=0}^{N-1} (\sum_{j=0}^{i-1} \xi_j
B^{(l)}_{t_j,t_{j+1}}-\xi_i B^{(l)}_{t_i}) B^{(k)}_{t_i,t_{i+1}} +
\sum_{i=0}^{N-1}
\xi_i \int_{t_i}^{t_{i+1}} B^{(l)}_r dB^{(k)}_r \nonumber \\
&=& \sum_{i=0}^{N-1} \tilde{\xi}_i B^{(k)}_{t_i,t_{i+1}} +
\sum_{i=0}^{N-1} \xi_i \int_{t_i}^{t_{i+1}} B^{(l)}_r dB^{(k)}_r, \
\ \ \hc-q.s.,\label{gs1}
\end{eqnarray}

On the other hand, suppose $\op:=\{0=\tau_0 < \tau_1 <... <
\tau_M=T\}.$ It holds that
\begin{eqnarray}
&&\sum_{j=0}^{M-1} Y_{\tau_k} B^{(k)}_{\tau_j,\tau_{j+1}} \nonumber \\
&=& \sum_{j=0}^{M-1}(\sum_{i=0}^{N-1}\tilde{\xi}_i 1_{[t_i,
t_{i+1})} (\tau_j) B^{(k)}_{\tau_j, \tau_{j+1}})  \nonumber \\
&+&\sum_{j=0}^{M-1} B^{(l)}_{\tau_j} (\sum_{i=0}^{N-1} \xi_i
1_{[t_i, t_{i+1})} (\tau_j)) B^{(k)}_{\tau_j, \tau_{j+1}}, \ \ \
\hc-q.s. \label{gs3}
\end{eqnarray}

We claim that the first part of (\ref{gs3}) converges to the first
part of (\ref{gs1}) in $L_G^1$-norm sense, and the second part of
(\ref{gs3}) also does converge
to the last part of (\ref{gs1}). \\

Firstly, for any $i=0,...,N-1,$ assume $\tau_{k_i}$ is the first
endpoint in partition $\op$ entering the interval $[t_i, t_{i+1}).$
Note that $k_i \geq 1,$ once making sure $|\op| < |\oq|.$ Then it
turns out that
\begin{eqnarray}
&&\sum_{j=0}^{M-1}(\sum_{i=0}^{N-1}\tilde{\xi}_i 1_{[t_i, t_{i+1})}
(\tau_j) B^{(k)}_{\tau_j, \tau_{j+1}}) \nonumber \\
&=& \sum_{i=0}^{N-1} \tilde{\xi}_i B^{(k)}_{t_i, t_{i+1}}+
\sum_{i=0}^{N-1} (\tilde{\xi}_i B^{(k)}_{\tau_{k_i}, t_i} +
\tilde{\xi}_i B^{(k)}_{\tau_{k_{i+1}-1}, t_{i+1}}). \label{gs4}
\end{eqnarray}
A similar argument as Lemma \ref{b condi} shows that the second part
of (\ref{gs4}) converges to $0$ in the $L_G^1$-norm sense.

The convergence of the second part of (\ref{gs3}) follows from the
fact that
$$
L_G^2-\sum_{\op \bigcap [t_i,t_{i+1})} B^{(l)}_u B^{(k)}_{u,v}
\stackrel{|\op|}{\rightarrow} \int_{t_i}^{t_{i+1}} B^{(l)}_r
dB^{(k)}_r.
$$

Step2. According to the definition of $M_G^2,$ for any
$Y_t:=\int_0^t \beta_s dB^{(l)}_s,$ with $\beta \in M_G^2, $ there
exists $\{\beta^n\}_{n=1} \in M_G^{2,0},$ such that $\beta^n
\stackrel{M_G^2}{\rightarrow} \beta.$ Then one has the following
identity by inserting terms
\begin{eqnarray}
&&\sum_\op Y_u B^{(k)}_{u,v} - \int_0^T Y_t dB^{(k)}_t \nonumber \\
 &=& \sum_\op (Y_u B^{(k)}_{u,v}- \int_0^u \beta_s^n dB^{(l)}_s)B^{(k)}_{u,v} \label{gs5} \\
 &+& \sum_\op \int_0^u \beta_s^n dB^{(l)}_s B^{(k)}_{u,v} - \int_0^T Y_t^n
 dB^{(k)}_t \label{gs6}\\
&+& \int_0^T Y_t^n dB^{(k)}_t - \int_0^T Y_t  dB^{(k)}_t,  \ \ \
\hc-q.s., \label{gs7}
\end{eqnarray}
where we denote $Y_t^n:= \int_0^t \beta_s^n dB^{(l)}_s.$\\

 We claim that (\ref{gs5}), (\ref{gs6}), (\ref{gs7}) converge to $0$ in the sense of $L_G^1.$\\

Firstly, for (\ref{gs7}), it suffices to prove that $Y_t^n
\stackrel{M_G^2}{\rightarrow} Y_t,$ which follows directly from
\begin{eqnarray*}
\E\int_0^T|\int_0^t (\beta_s^n-\beta_s)dB^{(l)}_s|^2 dt \leq
\bar{\sigma}^2 T \E\int_0^T |\beta_s^n-\beta_s|^2 ds \rightarrow 0
\end{eqnarray*}
as $n$ goes to infinity.

Secondly, for a fixed $n,$ according to Step1, (\ref{gs6}) converges
to $0$ as $|\op| \rightarrow 0.$

Thirdly, for (\ref{gs5}), it holds that
\begin{eqnarray*}
\E|\sum_\op(Y_u- \int_0^u \beta_s^n dB^{(l)}_s) B^{(k)}_{u,v}|^2&=&
\E\sum_\op
|Y_u-\int_0^u \beta_s^n dB^{(l)}_s |^2 |B^{(k)}_{u,v}|^2\\
&\leq & \bar{\sigma}^2 \sum_\op |v-u| \E | \int_0^u (\beta_s^n
-\beta_s)^2 d\langle B^{(l)}\rangle_s |\\
& \leq &  \bar{\sigma}^4 T \E\int_0^T |\beta_s^n -\beta_s|^2 ds
\rightarrow 0
\end{eqnarray*}
\end{proof}

\begin{coro}\quad
Suppose $Y_t$ defined as $(\ref{G-ito}),$ with $\beta,\alpha,\gamma
\in M_G^2.$ Then it holds that
$$
L_G^1-\lim_{|\op|\rightarrow 0} \sum_{(u,v) \in \op}
\frac{Y_u+Y_v}{2} B^{(k)}_{u,v}= \int_0^T Y_s \circ d B^{(k)}_s.
$$
\end{coro}

\begin{proof}
By the above proposition and linearity of integration, it suffices
to show the convergence of $\sum_\op Y_u B^{(k)}_{u,v}$ to $\int_0^T
Y_t dB^{(k)}_t,$ in the case that $Y_t=\int_0^t \alpha_s ds.$
Indeed, with an application of Fubini's theorem, one has
inequalities
\begin{eqnarray*}
&& \E|\int_0^T Y_t dB^{(k)}_t - \sum_\op Y_u B^{(k)}_{u,v}| \leq
\E|\sum_\op \int_u^v (\int_u^t \alpha_sds)dB^{(k)}_t |\\
&\leq& \sup_{P\in \mathscr{P}} \bar{\sigma} \sum_\op[
E_P|\int_u^v(\int_u^t \alpha_s ds)^2
dt|]^{\frac{1}{2}}\\
& \leq & \sup_{P\in \mathscr{P}} \bar{\sigma}
\sum_\op[E_P\int_u^v(\int_u^t \alpha^2_s ds)(t-u)dt]^\frac12\\
& \leq & \bar{\sigma} \sup_{P\in \mathscr{P}} T^\frac12 [\sum_\op
E_P \int_u^v(\int_u^t \alpha^2_s ds )\frac{t-u}{v-u}dt]^\frac12\\
&\leq & \bar{\sigma} T^\frac12 \sup_{P\in \mathscr{P}} [\sum_\op
E_P\int_u^v \alpha_s^2 (\frac{v-u}{2}-\frac{(s-u)^2}{2(v-u)})
ds]^\frac12\\
&\leq & \bar{\sigma} T^\frac12 |\op|^\frac12 \|\alpha\|_{M_G^2},
\end{eqnarray*}
which implies the expected conclusion.

\end{proof}

\begin{rem}\quad
Of course one can further consider the quadratic variation of two
$G$-It\^{o} processes, and obtain similar results. However, by now,
we already have got enough information to consider Stratonovich
integrals as rough integrals.

\end{rem}

%

In the case where $Y_s=B_s,$ one may define the Stratonovich
integral with respect to $G$-Brownian motion,
$$
\B_{s,t}^{strat} := \int_s^t B_{s,u} \circ dB_u = \B_{s,t} + \frac12
\langle B \rangle_{s,t}.
$$
Here $\langle B \rangle=\{\langle B^{(i)},B^{(j)}\rangle\}_{i,j}$ is
the variation matrix. According to Theorem \ref{Kol},
$\BB^{strat}:=(B,\B^{strat})$ is also quasi-surely rough paths.

\begin{coro}({\bf\small\itshape{$G$-Stratonovich integral as rough integral}})\quad

Assume $(Y,Y')(\omega) \in
\FD^{2\alpha}_{B(\omega)}([0,T],\mathcal{L}(\R^d,\R^n)),$
$\hc-q.s.,$ and $Y,Y' \in M_G^2(0,T),$ with values $Y_t,Y'_t$ in
$L_G^2(\Omega_t), $ for any $t\in[0,T].$ Furthermore, suppose
$\|\|Y\|_\alpha\|_{\mathbb{L}^2}, \|\|Y'\|_\alpha\|_{\mathbb{L}^2} ,
\|\|R^Y\|_{2\alpha}\|_{\mathbb{L}^2} < \infty.$ Then one has the
identity,
$$
\langle Y, B\rangle_t = \int_0^t Y'_s d\langle   B\rangle_s, \ \ \
\hc-q.s..
$$
Moreover, it holds that
$$
\int_0^t Y_s d\BB^{strat} = \int_0^t Y_s \circ dB_s, \ \ \ \hc-q.s..
$$
In particular, the rough integral $\int_0^t Y_s d\BB^{strat}$
belongs to $L_G^1.$

\end{coro}

\begin{proof}
Note that
$$
\sum_{(u,v)\in \op} Y_{u,v}(B_{u,v})=\sum_{(u,v)\in \op}(Y'_u
B_{u,v})(B_{u,v}) + \sum_{(u,v)\in \op} R_{u,v}^Y (B_{u,v}).
$$
By similar tricks applied in the proof of Lemma \ref{b condi} and
integrability of $\|Y'\|_\alpha, \|R^Y\|_{2\alpha},$ one could
obtain that
$$
\sum_{(u,v)\in \op}(Y'_u B_{u,v})(B_{u,v}) \rightarrow \int_0^t Y'_s d\langle B \rangle_s; \\
\sum_{(u,v)\in \op} R_{u,v}^Y (B_{u,v})  \rightarrow 0
$$
in the sense of $L_G^2.$ Then we got the existence of $\langle Y,
B\rangle,$ i.e. the following identity,
$$
\langle Y, B \rangle_t=\int_0^t Y'_s d\langle B \rangle_s, \ \ \
\hc-q.s..
$$
By the definition of $\BB^{strat}$ and rough integrals, it holds
that
$$
\int_0^t Y_s d\BB^{strat} = \int_0^t Y_s d\BB_s + \int_0^t Y'_s
d\langle B \rangle_s, \ \ \ \hc-q.s..
$$
Then the conclusion follows.
\end{proof}



\section{Roughness of $G$-Brownian Motion and the Norris Lemma}

To build the Norris lemma in $G$-framework through rough paths, we
need to show the $\theta $-H\"{o}lder roughness of $G$-Brownian
motion, i.e. $\hc(L_\theta(B)=0)=0,$ for any $\theta
> \frac12.$ The main idea for the proof of the result (i.e. Proposition \ref{theta}) is adapted from Proposition 6.11 in
Chapter 6 of \cite{FH14}.

\begin{lem}({\bf\small\itshape{exponential inequality}})\label{exe}\quad
Suppose $B_t$ be a d-dimensional $G$-Brownian motion. One has the
following inequality
\begin{equation}
\hc(\sup_{[0,T]}|B_t| \geq \frac{1}{\vep}) \leq d
\exp(-\frac{1}{\vep^2 d T \bar{\sigma}^2 })
\end{equation}
\end{lem}

\begin{proof}
By the representation for $\E$, it holds that
\begin{eqnarray*}
\hc(\sup_{[0,T]} |B_t| \geq \frac{1}{\vep} ) &=& \sup_{\alpha \in
\mathcal {A}^\Gamma  } P_0(\sup_{[0,T]} |\int_0^t \alpha_s dB_s|^2
\geq
\frac{1}{\vep^2})\\
&\leq & \sum_i \sup_{\alpha \in \mathcal {A}^\Gamma  }
P_0(\sup_{[0,T]} |\sum_j
\int_0^t \alpha_s^{(i,j)} dB_s^j|^2 \geq \frac{1}{d \vep^2} )\\
&\leq & d \exp({-\frac{1}{\vep^2 d T \bar{\sigma}^2 }}),
\end{eqnarray*}
where $P_0 $ is the Wiener measure, and classical Bernstein
inequality (see p.153 in \cite{RY99} for example) is applied in the
last inequality.
\end{proof}

\begin{rem}\quad
About large deviation results in $G-$framework, we refer readers to
\cite{GJ10} for more details.

\end{rem}

The $\theta$-H\"{o}lder roughness of the classical Brownian motion
was proposed and proved in \cite{HP13}, which gives a quantitative
version of the true roughness of Brownian motion, i.e.,
$$
\overline{\lim_{t\rightarrow s}} \frac{|B_{s,t}|}{|t-s|^\theta } =
\infty, \ \  a.s. \ ,
$$
when $\theta > \frac12$ (see \cite{FS13} for the definition of true roughness).\\

\begin{lem}\label{6.12}\quad
Let $B_t$ be a d-dimensional $G$-Brownian motion. Then there exists
positive constants $b, A,$ depending only on the dimension $d,$ such
that for any $\vep \in (0, 1),$ one has the bound
\begin{equation}
\hc(\inf_{|a|=1 \atop a \in \R^d} \sup_{t \in [0,T]} |(a \cdot B_t)|
\leq \vep) \leq A (\exp(-b T \underline{\sigma}^2 \vep^{-2}) +
\exp(-b T^{-1} (\bar{\sigma}  \vep)^{-2})).
\end{equation}

\end{lem}

\begin{proof}
Note that $B_t^a:=a \cdot B_t$ is a $G_{aa^T}-$ Brownian motion,
with $\bar{\sigma}_{aa^T}^2 = 2G(aa^T) =\E[a^T \langle B \rangle_1
a] \geq \underline{\sigma}^2 |a|^2 = \underline{\sigma}^2.$ Here
$\underline{\sigma}$ is positive as introduced in Part 2. According
to small ball estimates for $G$-Brownian motion, i.e. Lemma 6.1 in
\cite{Z15}, one has the bound
$$
\sup_{|a|=1} \hc(\sup_{t \in [0,T]}|(a \cdot B_t)| \leq \vep) \leq
\frac{4}{\pi} \exp(-\frac{T \pi^2 \underline{\sigma}^2}{8 \vep^2}),
$$
for any $\vep \in (0, 1).$ Now cover the sphere $|a|=1$ with at most
$D \vep^{-2d}$ balls of radius $\vep^2$ centered at $a_i,$ $D$ a
constant depending on how to divide the sphere or the ball. By
applying Lemma \ref{exe}, one obtains inequalities
\begin{eqnarray*}
\hc(\inf_{|a|=1} \sup_{[0,T]} |(a \cdot B_t)| \leq \vep) &\leq&
\sum_{i=1}^{D \vep^{-2d}} \hc(\inf_{ a \in O(a_i,\vep^2) \atop |a|=1
}
\sup_{t \in [0,T]} | (a \cdot B_t)| \leq \vep)\\
& \leq & D \vep^{-2d} [\sup_{|a|=1} \hc(\sup_{[0,T]}| (a \cdot B_t)|
\leq 2\vep) + \hc(\sup_{[0,T]}|B_t|
\geq \frac{1}{\vep})]\\
&\leq &  A (\exp(-b T \underline{\sigma}^2  \vep^{-2}) + \exp(-b
T^{-1} \bar{\sigma}^{-2}  \vep^{-2})).
\end{eqnarray*}

\end{proof}

\begin{prop}({\bf\small\itshape{H\"{o}lder roughness for $G$-Brownian motion}})\label{theta}\quad
Let $B$ be a d-dimensional $G$-Brownian motion. Then for any $\theta
\in (\frac12, 1),$ $B_.(\omega)$ is $\theta$-H\"{o}lder rough,
$\hc-q.s.$ with scale $\frac{T}{2}.$ More precisely, there exist
positive constants $K,l,$ depending on
$T,\bar{\sigma},\underline{\sigma},$ such that for any $\tilde{\vep}
\in (0, \frac{1}{2T^\theta} ),$ one has the bound
\begin{equation}
\hc(L_\theta (B) < \tilde{\vep}) \leq K \exp(-l \tilde{\vep}^{-2}).
\end{equation}

\end{prop}

\begin{proof}

Define $D_\theta (B):= \inf_{|a|=1, n\geq 1, k \leq 2^n} \sup_{s,t
\in [\frac{k-1}{2^n}T, \frac{k}{2^n}T]} 2^{\theta n} | (a \cdot
B_{s,t})|.$ Then for any fixed $a ,s,\vep ,$ with $|a|=1,$ $s \in
[0,T],$ and $\vep \in (0,\frac{T}{2}),$ there exist $n , k \in
\mathbb{N} ,$ such that $\frac{T}{2^n} < \vep \leq
\frac{T}{2^{n-1}},$ and $I_{k , n }:=[\frac{k -1}{2^n }T, \frac{k
}{2^{n }}T] \subset \{t: |t-s| \leq \vep\}.$ Moreover, by the
definition of $D_\theta(B),$ there exist $t_1, t_2 \in I_{k , n },$
such that
$$
| (a \cdot B_{t_1,t_2})| \geq 2^{-n \theta} D_\theta (B),
$$
so $t_1 $ or $t_2$(say $ t_1$) satisfies
$$
| (a \cdot B_{s,t_1})| \geq \frac12 2^{-n \theta} D_\theta (B) .
$$
According to the arbitrary choice of $a ,s,\vep ,$ it follows that
$$
L_\theta(B) \geq \frac12 \frac{2^{-n \theta}}{\vep^\theta}
D_\theta(B) \geq \frac12 (\frac{1}{2T})^\theta D_\theta(B).
$$

Finally, with an application of Lemma \ref{6.12}, one arrives at
inequalities
\begin{eqnarray*}
&& \hc(L_\theta(B) < \te) \leq \hc(D_\theta(B) < 2^{1+\theta}
T^\theta \te)\\
&\leq& \sum_{n=1}^{\infty} \sum_{k=1}^{2^n} \hc(\inf_{|a|=1}
\sup_{s,t \in I_{k,n}} | (a \cdot B_{s,t})| \leq
2^{-n\theta}2^{1+\theta}
T^\theta \te )\\
&\leq & \sum_{n=1}^{\infty} 2^n A[ \exp(-b T 2^{-n}
\underline{\sigma}^2 2^{2n\theta}
 (2^{1+\theta}T^\theta \te)^{-2})+ \exp(-b T^{-1} 2^{n}
\bar{\sigma}^{-2} 2^{2n\theta}
 (2^{1+\theta}T^\theta \te)^{-2})]   \\
 &\leq & \sum_{n=1}^{\infty} \tilde{A} \exp(-\tilde{b} n  (2^{1+\theta}T^\theta
 \te)^{-2})\\
 &\leq & K\exp(-l \te^{-2}),
\end{eqnarray*}

in the second last inequality of which, we apply the fact that there
exist positive constants $\tilde{A}, \tilde{b},$ depending on
$\bar{\sigma}, \underline{\sigma}$ and $T,$ such that
\begin{eqnarray*}
n \ln2 +  \tilde{b} n \bar{\vep}^{-2}  \leq  \ln \frac{\tilde{A}}{A}
+ b \bar{\vep}^{-2} (T \underline{\sigma}^2 2^{(2\theta-1)n} \wedge
 \bar{\sigma}^{-2} T^{-1}
2^{(2\theta+1)n} ),
\end{eqnarray*}
holds uniformly over $n \geq 1, \bar{\vep} \in (0, 1).$

\end{proof}

\begin{rem}\quad
According to the above proof, one could see the non-degenerateness
of $G$ is necessary. Furthermore, constants in the above bound are
uniform on the bounds of $\underline{\sigma}^2T$ and
$\bar{\sigma}^{-2}T^{-1}.$

\end{rem}

\begin{coro}\quad
Let $B_t$ a $d$-dimensional $G$-Brownian motion. Then it holds that,
for any $\theta > \frac12,$
\begin{equation}
\overline{\lim_{t\rightarrow s}} \frac{|B_{s,t}|}{|t-s|^\theta } =
\infty,\ \ \forall s \in [0,T], \ \ \ \hc-q.s.,
\end{equation}

\end{coro}

\begin{proof}
Indeed, one only needs to show the result in one-dimensional case.
For any $\theta
> \frac12,$ one can choose $\theta' $ such that $\frac12 < \theta' <
\theta.$ Note that $\hc(L_{\theta'} (B)=0) \leq \hc(L_{\theta'} (B)<
\vep), $ for any $\vep >0.$ According to the above proposition,
$L_{\theta'} (B(\omega))>0 , \ \hc-q.s..$ By the definition of
$L_{\theta'} (B(\omega)),$ it holds that, for any $s \in [0,T],$
$$
\overline{\lim_{t\rightarrow s}} \frac{|B_{s,t}|}{|t-s|^\theta }
\geq \overline{\lim_{t\rightarrow s}}L_{\theta'}
(B)\frac{|t-s|^{\theta'}}{|t-s|^{\theta}}=\infty.
$$

\end{proof}

\begin{example}\quad
Suppose $\bar{\sigma}>1,\underline{\sigma}<\frac12$ and $P^1$ the
law of $\frac{B.}{2}$ under $P^0,$ where $B.$ is the one-dimensional
canonical process and $P^0$ is the Wiener measure. By the
representation theorem for $G$-expectation, one obtains $P^0,P^1 \in
\mathscr{P}.$ Fix any $t\in (0,T],$ and define a measurable set
$$
A=\{ \langle B \rangle_t =t\}.
$$
It is clear that $P^0(A)=1, P^1(A)=0,$ so $P^0,P^1$ are mutually
singular. Following classical methods, it is quite possible to show
that $B$ is $\theta$-H\"{o}lder rough $P^0-a.s.$ and $P^1-a.s..$
However, it is nontrivial to obtain a common null set by classical
stochastic analysis. Note that the capacity $\hc$ could govern
infinitely many such mutually singular measures. This profit could
be quite advantageous when one faces practical problems involving
probability uncertainty.

\end{example}

\begin{coro}\quad
Let $\BB=(B, \B),$ $(Y,Y')(\omega) \in
\FD^{2\alpha}_{B(\omega)}([0,T],\mathcal{L}(\R^d,\R^n)), $ and $Z
\in \oc^\alpha([0,T],\R^n), \hc-q.s..$ Furthermore, suppose $(Y,Y')$
satisfies assumptions in Proposition \ref{itor}. Then denote
$I_t=\int_0^t Y_s dB_s+ \int_0^t Z_s ds,$ and $\mathcal {R}=
1+L_\theta(B)^{-1} + \|\BB\|_{\FC^\alpha} + \|Y,Y'\|_{B,2\alpha}
+|Y_0| +|Y'_0| +\|Z\|_\alpha +|Z_0|.$ One has the inequality
$$
\|Y\|_{\infty} +\|Z\|_{\infty} \leq M \mathcal{R}^q \|I\|_{\infty}
^r \ \ \ \hc-q.s.,
$$
for some constants $M,q,r,$ depending only on $\alpha, \theta, T.$
\\
In particular, if
$$
\int_0^t Y_s dB_s+ \int_0^t Z_s ds = \int_0^t Y'_s dB_s+ \int_0^t
Z'_s ds,
$$
it holds that $Y\equiv Y', \  Z\equiv Z', \ \ \ \hc-q.s..$
\end{coro}

\begin{proof}
For any fixed $\alpha ,$ there exists a constant $\theta \in
(\frac12, 2\alpha).$ According to Proposition \ref{theta}, $B$ is
$\theta$-H\"{o}lder rough, $\hc-q.s.$. By applying Theorem
\ref{NRP}, one could obtain the desired result.

\end{proof}

\begin{rem}\quad
According to the Norris lemma for rough paths, the above version of
Norris lemma in $G$-framework fails to distinguish the integral with
respect to $d\langle B \rangle$ and that with respect to $dt,$
mainly because as a quadratic variation process, $\langle B \rangle$
is no longer rough any more. The distinguish of integrals with
respect to $d\langle B \rangle$ and $ dt$ is done in \cite{S12} by
probabilistic methods. To give a quasi-surely quantitative
distinction between these two integrals, further work may need to be
done in the future.
\end{rem}


\Acknowledgements{This work was supported by National Natural
Science Foundation of China (Grant No. 10871011). S. G. Peng and H.
L. Zhang would thank valuable suggestions of F.L. Wang to this
paper.}


\end{document}